\RequirePackage{fix-cm}
\documentclass[smallextended]{svjour3}
\smartqed
\usepackage{latexsym,epsfig,amssymb,amsmath,color,url,bbm,pdfsync,tikz}
\usetikzlibrary{arrows,positioning} 
\usepackage{graphicx}
\usepackage[round, authoryear,semicolon,sort&compress]{natbib}
\usepackage{graphicx}
\allowdisplaybreaks
\numberwithin{equation}{section}
\numberwithin{figure}{section}
\numberwithin{theorem}{section}
\numberwithin{remark}{section}

\newtheorem{Theorem}{Theorem}
\newtheorem{Lemma}[Theorem]{Lemma}
\newtheorem{Remark}[Theorem]{Remark}
\newtheorem{Example}[Theorem]{Example}
\newtheorem{Proposition}[Theorem]{Proposition}
\newtheorem{Definition}[Theorem]{Definition}
\newtheorem{Corollary}[Theorem]{Corollary}

\def\be{\boldsymbol e}

\def\R{\mathbb{R}}

\newcommand{\bthe}{\begin{Theorem}}
\newcommand{\ethe}{\end{Theorem}}

\newcommand{\ble}{\begin{Lemma}}
\newcommand{\ele}{\end{Lemma}}

\newcommand{\bde}{\begin{Definition}}
\newcommand{\ede}{\end{Definition}}

\newcommand{\bco}{\begin{Corollary}}
\newcommand{\eco}{\end{Corollary}}

\newcommand{\bpr}{\begin{Proposition}}
\newcommand{\epr}{\end{Proposition}}

\newcommand{\brem}{\begin{Remark}}
\newcommand{\erem}{\end{Remark}}

\newcommand{\bexam}{\begin{Example}}
\newcommand{\eexam}{\end{Example}}

\newcommand{\beqq}{\begin{equation}}
\newcommand{\eeqq}{\end{equation}}

\newcommand{\beao}{\begin{eqnarray*}}
\newcommand{\eeao}{\end{eqnarray*}\noindent}

\newcommand{\beam}{\begin{eqnarray}}
\newcommand{\eeam}{\end{eqnarray}\noindent}

\newcommand{\barr}{\begin{array}}
\newcommand{\earr}{\end{array}}

\newcommand{\sid}[1]{{\color{black} #1}}
\newcommand\independent{\protect\mathpalette{\protect\independenT}{\perp}}
\def\independenT#1#2{\mathrel{\rlap{$#1#2$}\mkern2mu{#1#2}}}
\newcommand{\dd}{\mathrm{d}}
\newcommand{\PP}{\textbf{P}}
\newcommand{\EE}{\textbf{E}}
\newcommand{\ind}{\textbf{1}}

\newcommand{\convp}{\stackrel{P}{\rightarrow}}
\newcommand{\convas}{\stackrel{\text{a.s.}}{\longrightarrow}}
\newcommand{\RR}{\mathbb{R}}
\newcommand{\supy}{\sup_{y\ge 1}}

\newcommand{\argmin}{\operatornamewithlimits{argmin}}

\newcommand{\jtimei}{\tau^{(i)}}
\newcommand{\bfD}{\boldsymbol{\mathcal{D}}}

\newcommand{\wtd}[1]{{\color{black} #1}}

\begin{document}
\bibliographystyle{plainnat}

\title{Consistency of Hill Estimators in a Linear Preferential Attachment Model\thanks{This work was supported by Army MURI grant W911NF-12-1-0385 to Cornell University.}}
\titlerunning{Hill Estimator for Network Data}

\author{Tiandong Wang \and Sidney I. Resnick} 
\authorrunning{T. Wang \and S.I. Resnick}
\institute{Tiandong Wang \at
              School of Operations Research and Information Engineering, Cornell University,
Ithaca, NY 14853 \\
              \email{tw398@cornel.edu}           %  \\
           \and
           Sidney I. Resnick \at
              School of Operations Research and Information Engineering, Cornell University,
Ithaca, NY 14853 \\
\email{sir1@cornel.edu}
}
\date{Received: November 15, 2017}
\maketitle

\begin{abstract}
Preferential attachment is widely used to model power-law behavior of
degree distributions in both directed and undirected networks. 
Practical analyses on the tail exponent of the power-law degree distribution
use Hill estimator as one of the key summary statistics, whose consistency is justified mostly for iid data.
The major goal in this paper is to answer the question whether the Hill estimator is still consistent when applied to non-iid network data.
To do this, we first derive the asymptotic behavior of the degree sequence via embedding the degree growth of a fixed node into 
a birth immigration process. We also need to show the convergence of the tail empirical measure, from which
the consistency of Hill estimators is obtained. This step requires checking the concentration of degree counts. 
We give a proof for a particular linear preferential attachment model and use 
simulation results as an illustration in other choices of models.
%\noindent{\bf MSC Classes:} 60G70, 60B10, 60G55, 60G57, 05C80, 62E20.\\
\end{abstract}

\section{Introduction.}
The preferential attachment model gives a random graph in which nodes
and edges are added to the network based on probabilistic rules,  
and is used to \sid{mimic} the evolution of networks such as social networks, collaborator and citation networks, as well as recommender networks. The probabilistic rule depends on the node degree and captures the feature 
that nodes with larger degrees tend to attract more edges. 
Empirical analysis of social network data shows that degree
distributions follow power laws. Theoretically, this is
true  for linear preferential attachment models \sid{which} makes
preferential attachment a popular choice for network modeling
\citep{bollobas:borgs:chayes:riordan:2003, durrett:2010b,
  krapivsky:2001,krapivsky:redner:2001,vanderHofstad:2017}. 
The preferential attachment mechanism has been applied to both directed and undirected graphs. 
Limit theory for degree counts can be found in \citet{resnick:samorodnitsky:2016b}, \citet{bhamidi:2007}, \citet{krapivsky:redner:2001} for the undirected case and \citet{wang:resnick:2015}, \citet{resnick:samorodnitsky:towsley:davis:willis:wan:2016}, \citet{resnick:samorodnitsky:2015}, \citet{wang:resnick:2016}, \citet{krapivsky:2001} for the directed case.
This paper only focuses on the undirected case.

One statistical issue is how to estimate the index of the \sid{degree distribution} power-law
tail. In practice, this is often done by combining a minimum distance
method \cite{clauset:shalizi:newman:2009} with the Hill
estimator \sid{\cite{hill:1975}.} 
\sid{Data repositories of large network datasets such as KONECT (http://konect.uni-koblenz.de/)
\cite{kunegis:2013}  provide for each dataset key summary
statistics including Hill estimates of degree distribution tail
indices. However, there is no theoretical
justification for such estimates and}
consistency of the Hill estimator has been 
proved only for data from a stationary sequence of random variables,
which is assumed to be either iid \cite{mason:1982} or satisfy certain
structural or mixing 
assumptions, e.g. \cite{resnick:starica:1995, 
  resnick:starica:1998b, rootzen:leadbetter:dehaan:1990,hsing:1991}. 
%\sid{Of course
%existing theoretical justifications for the Hill procedure do not
%apply to degree distributions of network data.}
Therefore, proving/disproving the
consistency of Hill estimators for network data is a  major concern in this paper. 

\sid{The Hill estimator and other tail descriptors are often analyzed using
the tail empirical estimator. Using standard point measure notation,
let}
$$\epsilon_x(A)=\begin{cases} 1,& \text{ if }x\in A,\\ 0,& \text{ if
  }x\notin A \end{cases}.$$
For \sid{positive iid random variables} $\{X_i:i\geq 1\}$ whose
\sid{distribution has a regularly varying tail} with index $-\alpha<0$, 
we have the following convergence in the space of Radon measures on
$(0,\infty]$ of the  \sid{sequence of}  empirical measures
\beqq\label{conv1}
\sum_{i=1}^n\epsilon_{X_i/b(n)}(\cdot)\Rightarrow
\text{PRM}(\nu_\alpha (\cdot)), \quad\text{with }\quad\nu_\alpha(y,\infty] = y^{-\alpha},y>0,
\eeqq
to the limit Poisson random measure with mean measure
$\nu_\alpha(\cdot).$ 
From \eqref{conv1} other \sid{extremal properties of $\{X_n\}$}
follow \cite[Chapter 4.4]{resnick:1987}. \sid{See for example the
application 
given in this paper after Theorem \ref{thm:tail_meas}.}
Further, for \sid{any} intermediate sequence $k_n\to\infty$,
$k_n/n\to 0$ as $n\to\infty$, the sequence of tail empirical measures also
\sid{converge to a deterministic limit,}
\beqq\label{conv2}
\frac{1}{k_n}\sum_{i=1}^n\epsilon_{X_i/b(n/k_n)}(\cdot) \Rightarrow
\nu_\alpha (\cdot),\eeqq
which is \sid{one way} to prove consistency of the Hill estimator for iid data \citet[Chapter 4.4]{resnickbook:2007}.
\sid{We seek a similar dual pair as \eqref{conv1} and \eqref{conv2} for
network models that facilitates the study of the Hill estimator and
extremal properties of node degrees.}

With this goal in mind, we first find the limiting distribution for the degree sequence in a linear preferential attachment model,
from which a similar convergence result to \eqref{conv1} follows.
Embedding the network growth model into a continuous time branching process 
(cf. \citet{bhamidi:2007, athreya:2007, athreya:ghosh:sethuraman:2008}) is a useful tool in this case. 
We model the growth of the degree of each single node as a birth
process with immigration. Whenever a new node is added 
\sid{to} the network, a new birth immigration process is
initiated. In this \sid{embedding}, the total number of nodes in the network
growth model also forms a birth immigration process. \sid{Using} results
from the limit theory of continuous time branching processes
(cf. \citet[Chapter~5.11]{resnick:1992}; \citet{tavare:1987}), we
\sid{give} the limiting distribution of the degree of a fixed
node as well as the maximal degree growth. 

\sid{Empirical evidence for simulated networks 
lead\wtd{s} to the belief that the} Hill estimator is consistent. However,
proving the analogue of \eqref{conv2} is challenging and 
requires showing concentration inequalities for \sid{expected} degree
counts. We \sid{have} only succeeded for a particular linear preferential
attachment model, where each new node must attach to one of the
existing nodes in the graph. \sid{We are not sure the concentration
inequalities always hold for preferential attachment} and discussion of
limitations of the Hill estimator for network data must be left for the
future.
\sid{For a more sophisticated model where we could not verify the
concentration inequalities}, we illustrate consistency of the Hill
estimator coupled with \wtd{a minimum distance method (introduced in \cite{clauset:shalizi:newman:2009})} via
simulation for a range of parameter values; however the
asymptotic distribution of the Hill estimator in this case is
confounding and it is not obviously normal.
\sid{Whether this possible non-normality is due to the minimum distance
threshold selection or due to network data (rather than iid data)
being used, we are not sure at this point.}

The rest of the paper is structured as follows. \sid{After giving
background
on the tail empirical measure and Hill estimator in the rest of this section},
Section~\ref{sec:motiv} gives two linear preferential attachment
models.
Section~\ref{sec:prelim} summarizes  \sid{key facts about the pure birth
 and the birth-immigration processes.}
We analyze  \sid{social network} degree growth in Section~\ref{sec:embed} using  a sequence of
birth-immigration processes and \sid{give} the limiting 
\sid{empirical measures of normalized degrees in the style of \eqref{conv1}}
for both models under consideration. We prove
the consistency of the Hill estimator for the simpler model in
Section~\ref{sec:Hill} and give simulation results in Section~\ref{sec:sim}
that illustrate  the behavior of Hill estimators in the
other model. 

\sid{Parameter estimation based on maximum likelihood or approximate
  MLE for {\it directed\/}
preferential attachment models is studied in
\cite{wan:wang:davis:resnick:2017}.  A comparison between MLE model
based methods and asymptotic extreme value methods is forthcoming.}

\subsection{Background}
\sid{Our approach to the Hill estimator considers it as a functional
of the tail empirical measure so we start with necessary background}
and review  standard results (cf. \citet[Chapter 3.3.5]{resnickbook:2007}).  

For $\mathbb{E} = (0,\infty]$, let $M_+(\mathbb{E})$  be the set
of  non-negative Radon measures on $\mathbb{E}$. 
A point measure $m$ is an element of $M_+(\mathbb{E})$ of the form
\beqq\label{eq:pm}
m = \sum_i\epsilon_{x_i}.
\eeqq
The set $M_p(\mathbb{E})$ is the set of all Radon point measures of the form \eqref{eq:pm}
and $M_p(\mathbb{E})$ is a closed subset of $M_+(\mathbb{E})$ in
the vague metric.

For $\{X_n, n\geq 1\}$ iid  and non-negative with common
\sid{regularly varying} distribution
\sid{tail} $\overline{F}\in RV_{-\alpha}$, $\alpha>0$, there  
exists a sequence $\{b(n)\}$ such that  for a limiting Poisson random
measure with mean measure $\nu_\alpha $ and $\nu_\alpha(y,\infty] =
y^{-\alpha}$ for $y>0$, written \wtd{as} PRM($\nu_\alpha$), we have\beqq\label{eq:bnX}
\sum_{i=1}^n\epsilon_{X_i/b(n)}\Rightarrow \text{PRM}(\nu_\alpha),\quad\text{ in }M_p((0,\infty]),
\eeqq
and for some $k_n\to\infty$, $k_n/n\to 0$,
\beqq\label{eq:bnkX}
\frac{1}{k_n}\sum_{i=1}^n\epsilon_{X_i/b(n/k_n)}\Rightarrow \nu_\alpha,\quad\text{ in }M_+((0,\infty]),
\eeqq
\sid{Note the limit in \eqref{eq:bnX} is random while that in
\eqref{eq:bnkX} is deterministic.} 
\sid{Define the Hill estimator $H_{k,n}$ based on $k$ upper order statistics of
$\{X_1,\dots,X_n\}$} as in \cite{hill:1975}
\[
H_{k,n} := \frac{1}{k}\sum_{i=1}^{k}\log\frac{X_{(i)}}{X_{(k+1)}},
\]
where $X_{(1)}\ge X_{(2)}\ge \ldots\ge X_{(n)}$ are order statistics of $\{X_i:1\le i\le n\}$.
In the iid case there are many proofs of consistency (cf. \citet{mason:1982,mason:turova:1994,hall:1982,dehaan:resnick:1998,
  csorgo:haeusler:mason:1991a}):  For
$k=k_n\to\infty,\,k_n/n\to 0$, we have
\begin{equation}\label{e:Hillconv}
H_{k_n,n} \convp 1/\alpha\qquad\text{as }n\to\infty.
\end{equation}

\sid{The treatment in \citet[Theorem~4.2]{resnickbook:2007} approaches 
consistency by showing \eqref{e:Hillconv} follows from \eqref{eq:bnkX}
and we follow this approach for the network context where the iid case
is inapplicable.}
\sid{The next section constructs two undirected preferential attachment models,
labelled A and B, and gives behavior of $D_i(n)$, the degree of node
$i$ at the $n$th stage of construction.} Theorem~\ref{thm:tail_meas}
shows that for $\delta$, a parameter in the model construction, 
the degree sequences in either Model A or B have  empirical measures
\beqq\label{emp_meas}
\sum_{i=1}^n \epsilon_{D_i(n)/n^{1/(2+\delta)}}
\eeqq
that converge weakly to some \sid{random} limit point measure in $M_p((0,\infty])$.
The question then becomes whether there is an analogy to
\eqref{eq:bnkX} in the network
 case so that 
\beqq\label{eq:bnkD}
\frac{1}{k_n}\sum_{i=1}^n\epsilon_{D_i(n)/b(n/k_n)}\Rightarrow \nu_{2+\delta},\quad\text{ in }M_+((0,\infty]),
\eeqq
with some function $b(\cdot)$ and  intermediate sequence $k_n$. \sid{This
would facilitate proving consistency of the Hill estimator.}
We \sid{successfully} derive \eqref{eq:bnkD} for Model A in
Section~\ref{sec:Hill} and discuss why we failed  for Model B.
For Model A, we give the consistency of the Hill estimator.

\section{Preferential Attachment Models.}\label{sec:motiv}
\subsection{Model setup.}\label{subsec:PA}
We consider an undirected preferential attachment model initiated from
 the initial graph  $G(1)$, which  
consists of one node $v_1$ and a self loop. Node $v_1$ then has degree 2 at stage $n=1$. For $n\ge \sid{1}$,
we obtain a new graph $G(n+1)$ by appending a new node $v_{n+1}$ to the existing graph $G(n)$.
The graph $G(n)$ consists of $n$ edges and $n$ nodes. Denote the set
of nodes in $G(n)$ by $V(n) := \{v_1, v_2, \ldots, v_n\}$. 
For $v_i\in V(n)$,  $D_i(n)$ is the degree of $v_i$ in $G(n)$.
We consider two ways to construct the random graph and refer to them as Model A and B.\smallskip\\
\textbf{Model A}: 
Given $G(n)$,
the new node $v_{n+1}$ is connected to one of the existing nodes $v_i\in V(n)$ with probability
\beqq\label{eq:prob1}
\frac{f(D_i(n))+\delta}{\sum_{i=1}^n \left(f(D_i(n))+\delta\right)},
\eeqq
where \sid{{\it the preferential attachment function\/}} $f(j), j\ge 1$  is deterministic and \sid{non-decreasing}.
In this case, the new node $\sid{v_{n+1}}$ for $n\geq 1$, is always born with degree 1.\smallskip\\
\textbf{Model B}: In this model, 
given graph $G(n)$, the graph $G(n+1)$ is obtained by \sid{either:}
\begin{itemize}
\item \sid{Adding a new node $v_{n+1}$ and a new edge connecting} to \sid{an} existing node $v_i\in V(n)$ with probability
\beqq\label{eq:prob2}
\frac{f(D_i(n))+\delta}{\sum_{i=1}^n \left(f(D_i(n))+\delta\right) + f(1)+\delta},
\eeqq
where $\delta >-f(1)$ is a parameter;

\medskip
or
\medskip

\item \sid{Adding a new node $v_{n+1}$} with a self loop  with probability
\beqq\label{eq:prob3}
\frac{f(1)+\delta}{\sum_{i=1}^n \left(f(D_i(n))+\delta\right) + f(1)+\delta}.
\eeqq
\end{itemize}

\sid{{\it Linear case\/}:} If the preferential attachment function is $f(j) = j$ for
$j= 1,2,\ldots$, then \sid{the model is called} the \emph{linear preferential attachment model}. 
Since every time we add a node and an edge the degree of 2 nodes is
increased by 1, \sid{we have for both model A or B that}
$
\sum_{i=1}^{n} D_i(n) = 2n, \,n\ge 1.
$ 
\sid{Therefore,} the attachment probabilities in \eqref{eq:prob1}, \eqref{eq:prob2} and \eqref{eq:prob3} are
\[
\frac{D_i(n)+\delta}{(2+\delta)n},\quad\frac{D_i(n)+\delta}{(2+\delta)n + 1+\delta}, \quad\text{  and  }\quad  \frac{1+\delta}{(2+\delta)n + 1+\delta},
\]respectively, where $\delta>-1$ is a constant.

\subsection{Power-law tails.}
\sid{Continuing with $f(j)=j$,}  suppose $G(n)$ is a random graph generated by either Model A or B after $n$ steps. 
Let $N_k(n)$ be the number of nodes in $G(n)$ with degree equal to $k$, i.e.
\beqq\label{eq:defN}
N_k(n) := \sum_{i=1}^n \ind_{\{D_i(n) = k\}}, 
\eeqq
then $N_{>k} (n) := \sum_{j>k} N_j(n)$, $k\ge 1$, is the number of nodes in $G(n)$ with degree strictly greater than $k$.
For $k= 0$, we set $N_{>0}(n) = n$. 

\sid{For both models A and B,} it \sid{is}
shown in \citet[Theorem~8.3]{vanderHofstad:2017} 
using concentration inequalities and martingale methods
that 
for fixed $k\ge \wtd{1}$, as $n\to\infty$,
\beqq\label{eq:pmfk}
\frac{N_k(n)}{n}\convp p_k =
(2+\delta)\frac{\Gamma(k+\delta)\Gamma(3+2\delta)}{\Gamma(k+3+2\delta)\Gamma(1+\delta)}\sid{\sim 
(2+\delta) \frac{\Gamma(3+2\delta)}{\Gamma(1+\delta)}k^{-(3+\delta)}; }
\eeqq
 $\left(p_k\right)_{k\ge 0}$ is a pmf \sid{and the asymptotic form, as
 $k\to\infty$,  follows
 from Stirling.}
\sid{Let $p_{>k}=\sum_{j>k}p_j$ be the complementary \wtd{cdf} and}
by Scheff\'e's lemma as well as
\citet[Equation (8.4.6)]{vanderHofstad:2017}, we have 
\beqq\label{eq:def_pk}
\frac{N_{>k}(n)}{n}\convp p_{>k} :=
\frac{\Gamma(k+1+\delta)\Gamma(3+2\delta)}{\Gamma(k+3+2\delta)\Gamma(1+\delta)}, 
\eeqq 
and again by Stirling's formula we get from  \eqref{eq:def_pk} as $k\to\infty$,
\[
p_{>k}\sim \sid{c\cdot } k^{-(2+\delta)}, \qquad  c=\frac{\Gamma(3+2\delta)}{\Gamma(1+\delta)}.
\]
In other words, the tail distribution of the asymptotic degree
sequence in a linear preferential attachment model is \sid{asymptotic}
to a power law with tail index $2+\delta$. 

In practice, the Hill estimator is widely used to estimate this tail
index.
\sid{Absent prior justification for using the Hill estimator on network
data, we investigate its use.}

\section{Preliminaries: Continuous Time Markov Branching Processes.}\label{sec:prelim}
In this section, we \sid{review} two continuous time Markov branching
processes \sid{needed} in Section~\ref{subsec:embed}, where we
embed the degree sequence of a fixed network node  into a
continuous time branching process and derive the asymptotic limit of
the degree growth. 

\subsection{Linear birth processes.}
%\bde\label{def:birth}
A linear birth process $\{\zeta(t):t\ge 0\}$ is a continuous time Markov process taking values in the set $\mathbb{N}^+ =\{1,2,3,\ldots\}$ and having a transition rate
$$ q_{i,i+1} = \lambda i, \qquad i\in \mathbb{N}^+, \quad\lambda> 0.$$
%\ede
The linear birth process $\{\zeta(t):t\ge 0\}$ is a mixed Poisson
process; \wtd{see
\citet[Theorem~5.11.4]{resnick:1992},
\citet{kendall:1966} and \citet{waugh:1971} among other sources.} If $\zeta(0)=1$ then the representation is
\beqq
\zeta(t)=1+N_0\bigl(W(e^{\lambda t} -1\bigr), \,t\geq 0,
\eeqq
where \wtd{$\{N_0(t): t\geq 0\}$ is a} unit rate homogeneous Poisson on $\R_+$ with
$N_0(0)=0$ and $W\independent N_0(\cdot)$ is a unit exponential random
variable independent of $N_0$. 
Since $N_0(t)/t\to 1$ almost surely \wtd{as $t\to\infty$}, it follows immediately that
\beqq \label{e:PBconv}
\frac{\zeta(t)}{e^{\lambda t}}\convas W,\qquad\text{as }t\to\infty.
\eeqq

We use these facts in Section~\ref{subsec:asy} to analyze
the asymptotic behavior of the degree growth in a preferential
attachment network. 

\subsection{Birth processes with immigration.}
Apart from  individuals within the population giving birth to new
individuals, population size can also increase due to 
immigration which is assumed independent of  births.
The linear birth process with immigration (B.I. process),
$\{BI(t): t\ge 0\}$, having lifetime parameter $\lambda>0$ and
immigration parameter $\theta\ge 0$ is a continuous time Markov
process with state space $\mathbb{N} =\{0,1,2,3,\ldots\}$ and
transition rate  
$$q_{i,i+1} = \lambda i + \theta.$$
When $\theta = 0$ there is no immigration and the B.I. process becomes
a pure birth process.

For $\theta>0$, the B.I. process starting from 0 can be
constructed from a Poisson process and an independent family of iid
linear birth processes \cite{tavare:1987}.
Suppose that $N_\theta (t)$ is the counting function of homogeneous
Poisson points $0<\tau_1<\tau_2<\ldots$ with rate
$\theta$ and independent of this Poisson process we have
 independent
copies of a linear birth process $\{\zeta_i(t):t\ge 0\}_{i\ge 1}$
with parameter $\lambda>0$ and $\zeta_i(0) = 1$ for $i\ge 1$.  
Let $BI(0) = 0$, then the  B.I. process is a shot noise process with form
\beqq\label{eq:defBI}
BI(t) := \sum_{i=1}^\infty
\zeta_i(t-\tau_i)\ind_{\{t\ge\tau_i\}}=\sum_{i=1}^{N_\theta(t) } \zeta_i(t-\tau_i).
\eeqq

\sid{
Theorem~\ref{thm:tavare} modifies slightly} the statement of 
\citet[Theorem~5]{tavare:1987}  summarizing  the asymptotic behavior of the B.I. process.
\begin{Theorem}\label{thm:tavare}
For $\{BI(t):t\ge 0\}$ as in \eqref{eq:defBI}, we have as $t\to\infty$,
\beqq \label{e:sigma}
e^{-\lambda t}BI(t) \convas \sum_{i=1}^\infty W_i e^{-\lambda\tau_i} \sid{=:\sigma}
\eeqq
where $\{W_i: i\ge 1\}$ are independent unit exponential random
variables satisfying for each $i\ge 1$,
   $$W_i=\lim_{t\to\infty} e^{-t}\zeta_i(t).$$
The random variable $\sigma$ in \eqref{e:sigma}
is a.s. finite and has a Gamma density given by
\[
f(x) = \frac{1}{\Gamma(\theta/\lambda)}x^{\theta/\lambda-1} e^{-x},\qquad x>0.
\]
\end{Theorem}

The form of $\sigma$  in \eqref{e:sigma} and its Gamma density is justified in
\cite{tavare:1987}. It can be guessed from
\eqref{eq:defBI} and some \wtd{cavalier} interchange of limits and infinite
sums. The density of $\sigma$ comes from transforming Poisson points
$\{(W_i,\tau_i), i \geq 1\}$, summing and recognizing a Gamma L\'evy
process at $t=1$.

\section{Embedding Process.}\label{sec:embed}
\sid{Our approach to the weak convergence of the \sid{sequence of empirical measures} in
\eqref{emp_meas} embeds the degree sequences $\{D_i(n), 1 \leq i
\leq n, n \geq 1\}$
 into a B.I. process. The embedding idea is proposed in
 \cite{athreya:ghosh:sethuraman:2008} and we tailor it for our setup
 finding it 
 flexible enough to accommodate both linear preferential attachment
 Models A and B introduced in Section~\ref{subsec:PA}.}

\subsection{Embedding.}\label{subsec:embed}
\sid{Here is how we embed the network growth model} using a sequence of independent B.I. processes.

\subsubsection{Model A and B.I. processes.}
 Model A is the simpler case where \sid{a new node is not allowed to
have self loop.}
Let $\{BI_i(t):t\ge 0\}_{i\ge 1}$ be independent B.I. processes such that
\beqq\label{eq:initial}
BI_1(0) = 2, \quad BI_i(0) = 1,\quad \forall i\ge 2.
\eeqq
\sid{Each has} transition rate is $\sid{q_{j,j+1}}=j+\delta$,
$\delta>-1$.
For  $i\ge 1$, let $\{\jtimei_k: k\ge 1\}$ be the jump times of
the B.I. process $\{BI_i(t): t\ge 0\}$
\sid{and set} $\jtimei_0 := 0$ for all $i\ge 1$. 
Then  for $k\ge 1$,
\[
BI_1(\tau^{(1)}_k) = k+2,\qquad BI_i(\jtimei_k) = k+1, \, i\ge 2.
\]
Therefore, 
\[
\tau^{(1)}_k - \tau^{(1)}_{k-1} \sim \text{Exp}(k+1+\delta),\quad\mbox{and }\quad
\jtimei_{k}-\jtimei_{k-1} \sim \text{Exp}(k+\delta), \, i\ge 2.
\]
and $\{\jtimei_{k}-\jtimei_{k-1}: i\ge 1, k\ge 1\}$ are independent.

Set $T_1^A=0$ and relative to  $BI_1(\cdot)$  define
\beqq\label{eq:defT2}
T^A_2:= \tau^{(1)}_1,
\eeqq
i.e. the first time that $BI_1(\cdot)$ jumps. 
Start \sid{the} new B.I. process $\{BI_2(t-T^A_2):{t\ge T^A_2}\}$ 
at $T^A_2$ and 
let $T^A_3$ be the first time after $T^A_2$ that either $BI_1(\cdot)$
or $BI_2,(\cdot)$ jumps so that,
\[
T^A_3 = \text{min}\{T^A_{i}+\jtimei_k: k\ge 1, T^A_{i}+\jtimei_k>T^A_2, i=1,2\}.
\]
Start a new, independent B.I. process $\{BI_3(t-T^A_3)\}_{t\ge T^A_3}$
at $T^A_3$. See  Figure~\ref{fig:embed}, \sid{which assumes} $
\tau_1^{(2)}+T_2^A< \tau_2^{(1)}$. 
Continue in this way. When $n$ lines have been created, define
$T^A_{n+1}$ to be the first time after $T^A_n$ that one of the processes
$\{BI_i(t-T^A_i): t\ge T^A_i\}_{1\le i\le n}$ jumps, i.e.
\beqq\label{eq:defTn}
T^A_{n+1} := \text{min}\{T^A_{i}+\jtimei_k: k\ge 1, T^A_{i}+\jtimei_k>T^A_{n}, 1\le i\le n\}.
\eeqq
At $T^A_{n+1}$, start a new, independent B.I. process
$\{BI_{n+1}(t-T^A_{n+1})\}_{t\ge T^A_{n+1}}$.

\newsavebox{\mytikzpic}
\begin{lrbox}{\mytikzpic} 
\begin{tikzpicture}[scale=1.15]
\begin{scope}
\draw [thick,->] (0,0) -- (9,0) node[below right] {$\scriptstyle t$};
\draw [thick,->] (2,-1) -- (9,-1) node[below right] {$\scriptstyle t$};
\draw [thick,->] (4,-2) -- (9,-2) node[below right] {$\scriptstyle t$};

%
%draw vertical lines
\foreach \x in {0,2,4}
\draw [very thick] (\x cm,3pt) -- (\x cm,-2pt);
\foreach \x in {6}
\draw (\x cm,3pt) -- (\x cm,-2pt);
\draw [thick](2, 4pt) -- (2,-1);
\draw [thick](4, 4pt) -- (4,-2);
\draw [very thick](2, -0.9) -- (2, -1.05);
\draw [very thick](4, -0.9) -- (4, -1.05);
\draw [very thick](4, -1.9) -- (4, -2.05);
\draw (7, -0.9) -- (7, -1.05);
\draw (0,0) node[above = 2pt] {$T^A_1 = 0$}
  node[above = 20pt]{$ BI_1(0)=2 $};
\draw (2,0) node[above = 2pt] {$T^A_2 = \tau_1^{(1)}$}node[above= 20pt]{$ BI_1(T^A_2)=3 $};
\draw (4.5,0) node[above = 2pt] {$T^A_3 = \tau_1^{(2)}+T_2^A$};
%\draw (5,0) node[above = 2pt]{$\tau_2^{(1)}$};
\draw (6,0) node[below = 2pt]{$\tau_2^{(1)}$};
\draw (8,0) node[below=3pt] {$ \cdots $} node[above=3pt] {$ \cdots $};
\draw (2,-1) node[below = 2pt]{$ BI_2(0)= 1$};
\draw (7,-1) node[below = 2pt]{$\tau_2^{(2)}+T^A_2$};
\draw (4, -2) node[below = 2pt]{$BI_1(T^A_3) = 3$}
                    node[below = 20pt]{$BI_2(T^A_3-T^A_2) = 2$}
                    node[below = 38pt]{$BI_3(0) = 1$};
\draw (8,-1) node[below=3pt] {$ \cdots $} node[above=3pt] {$ \cdots $};
\draw (8,-2) node[below=3pt] {$ \cdots $} node[above=3pt] {$ \cdots $};
\end{scope}
\end{tikzpicture}
\end{lrbox}
%\bigskip
  \begin{figure}
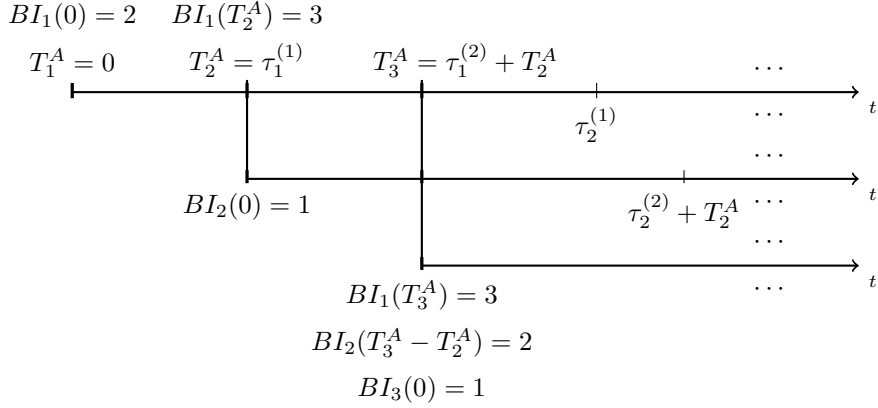

    \centering 
    \usebox{\mytikzpic} 
    \caption{Embedding procedure for Model A assuming $ \tau_1^{(2)}+T_2^A< \tau_2^{(1)}$.}\label{fig:embed}
\end{figure} 

\subsubsection{Model B and BI processes.}
In  Model B, \sid{a new node may} be born with a self loop
\sid{but} the B.I. process framework can \sid{still be used.}
We keep the independent sequence of $\{BI_i(t): t\ge 0\}_{i\ge 1}$ initialized as in \eqref{eq:initial},
as well as the definition of $\{\jtimei_k: k\ge 1\}$ for $i\ge 1$.

\sid{Set $T_0^B=T^B_1=0$} and start \emph{two} B.I. processes
$BI_1(\cdot)$ and $BI_2(\cdot)$ at $T^B_1$. 
At time $T^B_n$ with $n\ge 1$, there exist $n+1$ B.I. processes.
We define $T^B_{n+1}$ as the first time after $T^B_n$ that one of the processes
$\{BI_i(t-T^B_{i-1}): t\ge T^B_{i-1}\}_{1\le i\le n+1}$ jumps, i.e.
\beqq\label{eq:defTnb}
T^B_{n+1} := \text{min}\{T^B_{i-1}+\jtimei_k: k\ge 1, T^B_{i-1}+\jtimei_k>T^B_{n}, 1\le i\le n+1\},
\eeqq
and start a new, independent B.I. process $\{BI_{n+2}(t-T^B_{n+1})\}_{t\ge T^B_{n+1}}$ at $T^B_{n+1}$.

\subsubsection{Embedding.}
The following embedding theorem is similar to the one proved in
\cite{athreya:ghosh:sethuraman:2008} and summarizes how to embed in
the B.I. constructions.
\begin{Theorem}\label{thm:embed}  \sid{Fix $n\geq 1$.}\\
(a) For Model A, suppose 
$$\sid{\bfD^A(n):=\bigl(D_1^A(n),\dots,D_n^A(n) \bigr)}$$ \sid{is} the degree sequence of
nodes in the graph $G(n)$ 
and $\{T^A_n\}_{n\ge 1}$ is defined  as in \eqref{eq:defTn}.  For each fixed $n$, define
\begin{align*}
\widetilde{\bfD}^A(n) &:= (BI_1(T^A_n), BI_2(T^A_n-T^A_2), \ldots, BI_{n-1}(T^A_n-T^A_{n-1}), BI_n(0) ),
\end{align*}
and then $\bfD^A(n)$ and $\widetilde{\bfD}^A(n)$
%\begin{align*}
%\{D^A_i(n)\}_{1\le i\le n}, &\qquad \text{and}\qquad
 %\{BI_i(T^A_n-T^A_i)\}_{1\le i\le n}
%\end{align*}
have the same distribution in $\RR^n$.\\
(b) Analogously, for Model B, 
the degree sequence $$\bfD^B(n):=(D_1^B(n), \ldots, D_n^B(n))$$ and
 $$\widetilde{\bfD}^B(n) := (BI_1(T^B_n), BI_2(T^B_n), \ldots, BI_{n-1}(T^B_n-T^B_{n-2}), BI_n(T^B_n-T^B_{n-1}) )$$ 
 have the same distribution in $\RR^n$.
\end{Theorem}
\begin{proof}
By the construction of Model A, at each $T^A_n$, $n\ge 2$, we \sid{start} a new
B.I. process \sid{$BI_n(\cdot)$} with initial value equal to 1 and  
one of $BI_i$, $1\le i\le n-1$ also increases by 1. This makes the sum of the values of $BI_i$, $1\le i\le n$, increase by 2 so that
\[
\sum_{i=1}^n \left(BI_i(T^A_n-T^A_i)+\delta\right) = (2+\delta)n.
\] 
The rest is essentially the proof of
\citet[Theorem~2.1]{athreya:ghosh:sethuraman:2008} \sid{which we now outline.}

Both $\{ \bfD^A(n),n\geq 1\}$ and $\{\widetilde{\bfD}^A(n), n \geq 1\}$
are Markov on the state space $\cup_{n\geq 1}\R_+^n$ since
\begin{align*}
\bfD^A(n+1)=&\bigl(\bfD^A(n),1\bigr)  +\bigl(\be_{J_{n+1}}^{(n)},0\bigr),\\
\widetilde{\bfD}^A(n+1)=&\bigl(\widetilde{\bfD}^A(n),1\bigr)+\bigl(\be_{L_{n+1}}^{(n)},0\bigr),
\end{align*}
where for $n
\geq 1,$ $\be^{(n)}_j$ is a vector of length $n$ of $0$'s except for a
$1$ in the $j$-th \wtd{entry} and
$$P[J_{n+1}=j|\bfD^A(n)]=\frac{D^A_j(n)+\delta}{(2+\delta)n},\qquad 1
\leq j \leq n,$$
and $L_{n+1}$ records which B.I. process in $\{BI_i(t-T_i^A): t\ge T_i^A\}_{1\le i\le n}$ is the first
to have a new birth after $T^A_n$.
%has a similar structure discussed below in a few lines.

%The Markov property 
%of $\widetilde{\bfD}^A(n)$ follows from the strong Markov property of
%$\{BI(t): t\ge 0\}$. 
When $n =1$, 
\[
\widetilde{\bfD}^A(1) = BI_1(0) = 2 = D_1^A(1) = \bfD^A(1),
\]
\sid{so to prove equality in distribution for any $n$,}
it suffices to verify that the transition probability from
$\widetilde{\bfD}(n)$ to $\widetilde{\bfD}^A(n+1)$ is the same as that
from 
${\bfD}^A(n)$ to ${\bfD}^A(n+1)$.

According to the preferential attachment setup, we have
\begin{align}
\PP \Big({\bfD}^A(n+1) &=(d_1, d_2, \ldots, d_i+1, d_{i+1},\ldots, d_n, 1)\Big\vert {\bfD}^A(n) = (d_1, d_2, \ldots, d_n)\Big)\nonumber\\
&= \frac{d_i+\delta}{(2+\delta)n}, \quad 1\le i \le n.\label{transprob1}
%\PP\Big({\bfD}^A(n+1) &= (d_1, d_2, \ldots, d_n, 2,0,\ldots)\Big\vert {\bfD}^A(n) = (d_1, d_2, \ldots, d_n, 0,0,\ldots)\Big)\label{transprob2}\\
%&= \frac{1+\delta}{(2+\delta)n + 1+\delta}.
\end{align}
At time $T_n^A$, there are $n$ B.I. processes and each of them has a population size of $BI_i(T_n^A-T_i^A)$, $1\le i\le n$.
Therefore, $T_{n+1}^A- T_n^A$ is the minimum of $n$ independent exponential random variables, $\{E^{(i)}_{n}\}_{1\le i\le n}$, with means 
\[
\left(BI_i(T_n^A-T_i^A)+\delta\right)^{-1},\quad 1\le i\le n,
\]
which gives for any $1\le i \le n,$
\begin{align*}
\PP&\Big(L_{n+1}=i\Big\vert \widetilde{\bfD}^A(n) = (d_1, d_2, \ldots, d_n)\Big)\\
=&\PP \Big(\widetilde{\bfD}^A(n+1) =(d_1, d_2, \ldots, d_i+1, d_{i+1},\ldots, d_n, 1)\Big\vert \widetilde{\bfD}^A(n) 
= (d_1, \ldots, d_n)\Big)\\
=& \PP\Big(E^{(i)}_{n}<\bigwedge_{j=1, j\neq i}^{n} E^{(j)}_{n}\Big\vert\widetilde{\bfD}^A(n) = (d_1, d_2, \ldots, d_n)\Big)\\
=& \frac{BI_i(T_n^A-T_i^A)+\delta}{\sum_{i=1}^n \left(BI_i(T^A_n-T^A_i)+\delta\right)}=
\frac{d_i+\delta}{(2+\delta)n}.
%\PP\Big(\bfI(T_{n+1}) &= (d_1, d_2, \ldots, d_n, 2,0,\ldots)\Big\vert \bfI(T_n) = (d_1, d_2, \ldots, d_n, 0,0,\ldots)\Big)\\
%&= \PP\Big(E^{(n+1)}_{n+1}<\bigwedge_{j=1}^{n} E^{(j)}_{n+1}\Big\vert\bfI(T_n) = (d_1, d_2, \ldots, d_n, 0,0,\ldots)\Big)\\
%&= \frac{1+\delta}{(2+\delta)n + 1+\delta},
\end{align*}
This agrees with the transition probability in \eqref{transprob1}, thus completing the proof for Model A.

For Model B, the proof follows in a similar way except that for each $n\ge 1$, $T_{n+1}^B- T_n^B$ is 
the minimum of $n+1$ independent exponential random variables with means 
\[
\left(BI_i(T_n^B-T_{i-1}^B)+\delta\right)^{-1},\quad 1\le i\le n+1,
\]
so that for $1\le i\le n$,
\begin{align*}
\PP \Big(\widetilde{\bfD}^B(n+1) &=(d_1, d_2, \ldots, d_i+1, d_{i+1},\ldots, d_n, 1)\Big\vert \widetilde{\bfD}^B(n) 
= (d_1, d_2, \ldots, d_n)\Big)\\
&= \frac{BI_i(T_n^B-T_{i-1}^B)+\delta}{\sum_{i=1}^{n+1} \left(BI_i(T_n^B-T_{i-1}^B)+\delta\right)}=
\frac{d_i+\delta}{(2+\delta)n+1+\delta},\\
\intertext{and}
\PP\Big(\widetilde{\bfD}^B(n+1) &= (d_1, d_2, \ldots, d_n, 2)\Big\vert \widetilde{\bfD}^B(n) = (d_1, d_2, \ldots, d_n)\Big)\\
&= \frac{1+\delta}{(2+\delta)n + 1+\delta}.
\end{align*}
\end{proof}

\begin{Remark}{\rm
%The B.I. process construction for Model B can be interpreted as the following assumption on the network growth model.
%We actually assume that a new node $v_i$, $i\ge 2$, is born at stage $i-1$ with 
%degree 1, without attaching to any nodes. This node $v_i$ is not fully incorporated into the network model until stage $i$ by
%either attaching to one of the existing nodes or to itself (so that it is born with a self loop).

This B.I. process construction can also be generalized for other choices of the preferential attachment functions $f$.
For example, its applications to the super- and sub-linear preferential attachment models are studied in \cite{athreya:2007}.}
\end{Remark}

\subsection{Asymptotic properties.}\label{subsec:asy}
One important reason to use the embedding technique specified in Section~\ref{subsec:embed} is that asymptotic behavior of the degree growth in a preferential attachment model can be characterized explicitly.
These asymptotic properties then help us derive weak convergence of the empirical measure, which is analogous to \eqref{eq:bnX} in the iid case.

\subsubsection{Branching times.}
We first consider the asymptotic behavior of the branching times $\{T^A_n\}_{n\ge 1}$ and $\{T^B_n\}_{n\ge 1}$, which are defined in Section~\ref{subsec:embed}.
\begin{Proposition}\label{prop:asy_Tn}
For $\{T^A_n\}_{n\ge 1}$ and $\{T^B_n\}_{n\ge 1}$ defined in \eqref{eq:defTn} and \eqref{eq:defTnb} respectively, we have
\begin{align}
\frac{n}{e^{(2+\delta)T^A_n}}\convas W_A,\qquad W_A &\sim \text{Exp}\left(1\right);\label{eq:TAn_asy}\\
\intertext{ and }
\frac{n}{e^{(2+\delta)T^B_n}}\convas W_B,\qquad W_B &\sim \text{Gamma}\left(\frac{3+2\delta}{1+\delta}, 1\right).\label{eq:TBn_asy}
\end{align}
\end{Proposition}
\begin{proof}
Define two counting processes
\[
N_A(t) := \frac{1}{2}\sum_{i=1}^\infty BI_i(t-T^A_{i})\ind_{\{t\ge T^A_i\}}
\]
in Model A, and 
\[
N_B(t) := \frac{1}{2}\sum_{i=1}^\infty BI_i(t-T^B_{i-1})\ind_{\{t\ge T^B_i\}}
\]
in Model B.
In either case, we have 
\[
N_l(t)\ind_{\big\{t\in[T^l_n, T^l_{n+1})\big\}} =  n, \quad l= A,B.
\]
In other words, $\{T^l_n\}_{n\ge 1}$ are the jump times of the counting process $N_l(\cdot)$, for $l=A,B$, with the following structure
\begin{align}
\{T^A_{n+1}- T^A_{n}: n\ge 1\} &\stackrel{d}{=} \left\{\frac{A_i}{(2+\delta)i}, i\ge 1\right\},\label{eq:NA}\\
\{T^B_{n+1}- T^B_{n}: n\ge 1\} &\stackrel{d}{=} \left\{\frac{B_i}{(2+\delta)i+1+\delta}, i\ge 1\right\},\label{eq:NB}
\end{align}
where $\{A_i: i\ge 1\}$ and $\{B_i: i\ge 1\}$ are iid unit exponential random variables.

From \eqref{eq:NA}, we see that $N_A(\cdot)$ is a pure birth process with $N_A(0)=1$ and transition rate 
\[
q^A_{i,i+1} = (2+\delta)i ,\quad i\ge 1.
\]
\sid{Replacing $t$ with $T_n^A$} in  \eqref{e:PBconv}
 gives \eqref{eq:TAn_asy}. 
By \eqref{eq:NB}, $N_B(\cdot)$ is a B.I. process with $N_B(0)=1$ and
transition rate $q^B_{i,i+1}=(2+\delta)i+1+\delta$, $i\ge 1$. 
In order to apply Theorem~\ref{thm:tavare} which assumes $N_B(0)$, we
define $N_B'(t):= N_B(t)-1$ for all $t\ge 1$. Then $N'_B$ is a
B.I. process with $N_B'(0)=0$ and transition rate  
$(2+\delta)i+3+2\delta$, for $i\ge 0$. 
Therefore, \eqref{eq:TBn_asy} follows directly from Theorem~\ref{thm:tavare}.
\end{proof}

%\subsubsection{Degree sequences.}
\subsubsection{Convergence of the measure.}
Using embedding techniques, Theorem~\ref{thm:tail_meas} gives the convergence of the empirical measure, which draws an analogy to \eqref{eq:bnX} in the iid  case.
\begin{Theorem}\label{thm:tail_meas}
Suppose that 
\begin{enumerate}
\item[(1)]$\{T^l_i: i\ge 1\}$, $l=A,B$ are distributed as in \eqref{eq:NA} and \eqref{eq:NB}.
\item[(2)]$W_l$, $l=A,B$ are limit random variables as given in \eqref{eq:TAn_asy} and \eqref{eq:TBn_asy}.
\item[(3)]$\{\sigma_i\}_{i\ge 1}$ \wtd{is} a sequence of independent Gamma
  random variables specified in \eqref{eq:BI_asy} and
  \eqref{eq:sigma} below. 
\end{enumerate}
Then in $M_p((0,\infty])$, we have for $\delta\ge 0$,
\begin{subequations}\label{eq:meas_asy}
\begin{align}
\sum_{i=1}^n &\epsilon_{D^A_i(n)/n^{1/(2+\delta)}}(\cdot)\Rightarrow
               \sum_{i=1}^\infty \epsilon_{\sigma_i
               e^{-T^A_i}/W_A^{1/(2+\delta)}}(\cdot),\label{eq:meas_asyA}\\ 
\sum_{i=1}^n &\epsilon_{D^B_i(n)/n^{1/(2+\delta)}}(\cdot)\Rightarrow \sum_{i=1}^\infty \epsilon_{\sigma_i e^{-T^B_{i-1}}/W_B^{1/(2+\delta)}}(\cdot).\label{eq:meas_asyB}
\end{align}
\end{subequations}
\end{Theorem}

\begin{Remark}\label{rem:applic} {\rm
From \eqref{eq:meas_asyA} we get for any fixed $k\geq 1$, that in $\R_+^k$,
\begin{equation}\label{e:bigly}
\Bigl( \frac{D^A_{(1)}(n)}{n^{1/(2+\delta)}}, \dots
\frac{D^A_{(k)}(n)}{n^{1/(2+\delta)}}\Bigr) \Rightarrow
W_A^{-1/(2+\delta)}
\Bigl( 
(\sigma_\cdot
               e^{-T^A_\cdot} )_{(1)},\dots, (\sigma_\cdot
               e^{-T^A_\cdot} )_{(k)} \Bigr),
\end{equation}
where a subscript inside parentheses indicates ordering so that
$D_{(1)} ^A(n) \geq  \dots \geq D^A_{(k)}$ and the limit on the right
side of \eqref{e:bigly} represents the ordered $k$ largest points from  the
right side of \eqref{eq:meas_asyA}.
A similar result for Model B follows from \eqref{eq:meas_asyB}.
}
\end{Remark}

To prove Theorem~\ref{thm:tail_meas}, we first need to show the following lemma, which gives the
asymptotic limit of the degree sequence under the B.I. process framework.
\begin{Lemma}\label{lemma:degseq}
Suppose that 
\begin{enumerate}
\item[(1)]$\{T^l_i: i\ge 1\}$, $l=A,B$ are distributed as in \eqref{eq:NA} and \eqref{eq:NB}.
\item[(2)]$W_l$, $l=A,B$ are limit random variables as given in \eqref{eq:TAn_asy} and \eqref{eq:TBn_asy}.
\end{enumerate}
Then we have the following convergence results \sid{pertinent to} the degree sequence $\{D_i^l(n):1\le i\le n\}$, for $l=A,B$:
\begin{enumerate}
\item[(i)] 
For each $i\ge 1$,
\begin{subequations}\label{eq:BI_asy}
\begin{align}
\frac{BI_i(T^A_n - T^A_i)}{e^{T^A_n}} &\convas \sigma_i e^{-T^A_i}, \label{eq:BI_asyA}\\
\frac{BI_i(T^B_n - T^B_{i-1})}{e^{T^B_n}} &\convas \sigma_i e^{-T^B_{i-1}},\label{eq:BI_asyB}
\end{align}
\end{subequations}
where $\{\sigma_i\}_{i\ge 1}$ are a sequence of independent Gamma random variables with 
\beqq\label{eq:sigma}
\sigma_1\sim \text{Gamma}(2+\delta,1),\quad\text{and} \quad \sigma_i\sim \text{Gamma}(1+\delta,1),\quad i\ge 2.
\eeqq
Furthermore, \sid{for $i\ge
1$,} $ \sigma_i\sid{\independent}  e^{-T^A_i}$ in Model A  and $\sigma_i\independent  e^{-T^B_{i-1}}$ in Model B. 
\item[(ii)] 
For $\delta>-1$,
\begin{subequations}\label{eq:max_asy}
\begin{align}
\max_{i\ge 1}\, &\frac{D^A_i(n)}{n^{1/(2+\delta)}}\,\sid{\convp}\,W_A^{-1/(2+\delta)}\max_{i\ge 1} \sigma_i e^{-T^A_i},\label{eq:max_asyA}\\
\max_{i\ge 1}\, &\frac{D^B_i(n)}{n^{1/(2+\delta)}}\,\sid{\convp}\, W_B^{-1/(2+\delta)}\max_{i\ge 1} \sigma_i e^{-T^B_{i-1}},\label{eq:max_asyB}
\end{align}
\end{subequations}
where we set $T^B_0:=0$ and $D^l_i(n):=0$ for all $i\ge n+1$, $l=A,B$.
\end{enumerate}
\end{Lemma}
\begin{proof}
(i) For the B.I. processes $\{BI_i(\cdot)\}_{i\ge 1}$ defined here, all of them have initial values greater than 0.
Hence, in order to apply the asymptotic results in \cite{tavare:1987}, we need to modify them such that they all start with 0.
To do this, set for all $t\ge 0$,
\[
BI'_1(t) := BI_1(t) - 2, \qquad BI'_i(t) := BI_i(t) - 1, \quad i\ge 2,
\]
and we have $BI'_i(0) = 0$ for all $i\ge 0$. The transition rate needs to be changed accordingly, i.e. the process $BI'_1(\cdot)$
has transition rate $q'_{j,j+1}=j+2+\delta$ and that for $BI'_i(\cdot)$, $i\ge 2$, becomes $ j+1+\delta$, $j\ge 0$.

Throughout the rest of the proof of Lemma~\ref{lemma:degseq}, we only show the case for Model A and the result for Model B follows from the same argument.
Now applying Theorem~\ref{thm:tavare} gives that as $t\to\infty$,
\begin{align*}
\frac{BI_i(t-T^A_{i})}{e^{t-T^A_{i}}} & \stackrel{\text{a.s.}}{\longrightarrow} \sigma_i ,\quad i\ge 1,
\end{align*}
where $\{\sigma_i\}_{i\ge 1}$ are independent Gamma random variables with 
\[
\sigma_1\sim \text{Gamma}(2+\delta,1)\quad\text{and} \quad \sigma_i\sim \text{Gamma}(1+\delta,1),\quad i\ge 2.
\]
Thus as $n\to\infty$,
\begin{align*}
\frac{BI_i(T^l_n-T^A_{i})}{e^{T^l_n-T^A_{i}}} & \stackrel{\text{a.s.}}{\longrightarrow} \sigma_i,\quad i\ge 1,
\end{align*}
which gives \eqref{eq:BI_asyA}.

For $i\ge 2$, the independence of $\sigma_i$ and $T^A_i$ follows from the construction and
%$$\sigma_i\in \bigcap_{t\ge T_{i}}\sigma\left(\bigcup_{s\ge t}\sigma\Big(BI_i(s-T_{i})\Big)\right).$$
this completes the proof of (i).\\
(ii) Combining \eqref{eq:BI_asyA} with \eqref{eq:TAn_asy}, we have for fixed $1\le i\le n$,
\[
\frac{BI_i(T^A_n-T^A_i)}{n^{1/(2+\delta)}}\convas \frac{\sigma_i e^{-T^A_i}}{W_A^{1/(2+\delta)}},
\]
and $BI_i(T^A_n-T^A_i) = 0$ for $i\ge n+1$.
By Theorem~\ref{thm:embed}, it suffices to show
\[
\max_{i\ge 1}\frac{BI_i(T^A_n-T^A_i)}{n^{1/(2+\delta)}}\convas \max_{i\ge 1}\frac{\sigma_i e^{-T^A_i}}{W_A^{1/(2+\delta)}},
\]
which is proved in \citet[Theorem~1.1(iii)]{athreya:ghosh:sethuraman:2008}.
\end{proof}

With the preparation in Lemma~\ref{lemma:degseq}, we are ready to prove the convergence result in Theorem~\ref{thm:tail_meas}.\\
\smallskip
%\begin{proof}[Proof of Theorem~\ref{thm:tail_meas}]
{\it Proof of Theorem~\ref{thm:tail_meas}.}
Note that the limit random variables $$\sigma_i e^{-T^A_i}W_A^{-1/(2+\delta)},\quad i\ge 1,$$ 
have continuous distributions, so for any $y>0$,
\[
\PP\left(\sum_{i=1}^\infty \epsilon_{\sigma_i e^{-T^A_i}/W_A^{1/(2+\delta)}}(\{y\}) = 0\right) = 1.
\]
Hence, by Kallenberg's theorem for weak convergence to a point process on an interval (see \citet[Theorem~4.18]{kallenberg:2017} and
\citet[Proposition~3.22]{resnick:1987}), proving \eqref{eq:meas_asyA} requires checking
\begin{enumerate}
\item[(a)] For $y>0$, as $n\to\infty$,
\beqq\label{eq:meas_cond1}
\EE \left(\sum_{i=1}^n \epsilon_{D^A_i(n)/n^{1/(2+\delta)}}(y,\infty]\right)
\to \EE\left( \sum_{i=1}^\infty \epsilon_{\sigma_i e^{-T^A_i}/W_A^{1/(2+\delta)}}(y,\infty]\right).%,\quad\text{as }n\to\infty.
\eeqq
\item[(b)] For $y>0$, as $n\to\infty$,
\begin{align}\label{eq:meas_cond2}
\PP&\left(\sum_{i=1}^n \epsilon_{D^A_i(n)/n^{1/(2+\delta)}}(y,\infty] = 0\right)\nonumber\\
&\longrightarrow \PP\left( \sum_{i=1}^\infty \epsilon_{\sigma_i e^{-T^A_i}/W_A^{1/(2+\delta)}}(y,\infty] = 0\right).%,\quad\text{as }n\to\infty.
\end{align}
\end{enumerate}

To show \eqref{eq:meas_cond1}, first note that for any $M>0$, 
\begin{align*}
\EE \left(\sum_{i=1}^M \epsilon_{D^A_i(n)/n^{1/(2+\delta)}}(y,\infty]\right)
& =\sum_{i=1}^M \PP\left(\frac{D^A_i(n)}{n^{1/(2+\delta)}}>y\right)\\
&\longrightarrow \sum_{i=1}^M\PP\left(\sigma_i e^{-T^A_i}W_A^{-1/(2+\delta)}>y\right)\\
& = \EE\left( \sum_{i=1}^M \epsilon_{\sigma_i e^{-T^A_i}/W_A^{1/(2+\delta)}}(y,\infty]\right),
\end{align*}
as $n\to\infty$.
By Chebyshev's inequality we have for any $k>2+\delta$,
\begin{align}
\EE \left(\sum_{i=M+1}^n \epsilon_{D^A_i(n)/n^{1/(2+\delta)}}(y,\infty]\right)
& =\sum_{i=M+1}^n \PP\left(\frac{D^A_i(n)}{n^{1/(2+\delta)}}>y\right)\nonumber\\
&\le y^{-k} \sum_{i=M+1}^n \EE\left[\left(\frac{D^A_i(n)}{n^{1/(2+\delta)}}\right)^k\right].\label{eq:tail_meansum}
\end{align}
Also, we have for $\delta\ge 0$,
\[
\EE\left[\left(\frac{D^A_i(n)}{n^{1/(2+\delta)}}\right)^k\right]\le \EE\left[\left(\frac{D^A_i(n)+\delta}{n^{1/(2+\delta)}}\right)^k\right]
\le \EE\left[\left(\frac{\sigma_i e^{-T^A_i}}{W_A^{1/(2+\delta)}}\right)^k\right],
\]
where the last inequality follows from the result in \citet[Equation~(8.7.26)]{vanderHofstad:2017}.
From \citet[Equation~(8.7.22)]{vanderHofstad:2017}, we have
\[
\EE\left[\left(\frac{\sigma_i e^{-T^A_i}}{W_A^{1/(2+\delta)}}\right)^k\right]
= \frac{\Gamma(i-\frac{1}{2+\delta})}{\Gamma(i+\frac{k-1}{2+\delta})}\frac{\Gamma(k+1+\delta)}{\Gamma(1+\delta)}
\sim C_{k,\delta} i^{-\frac{k}{2+\delta}},
\]
for $i$ large and $C_{k,\delta}>0$.
Hence, continuing from \eqref{eq:tail_meansum}, we have 
\begin{align*}
\EE \left(\sum_{i=M+1}^n \epsilon_{D^A_i(n)/n^{1/(2+\delta)}}(y,\infty]\right)
&\le y^{-k} \sum_{i=M+1}^n \EE\left[\left(\frac{D^A_i(n)}{n^{1/(2+\delta)}}\right)^k\right]\\
&\le y^{-k} \sum_{i=M+1}^\infty \EE\left[\left(\frac{\sigma_i e^{-T^A_i}}{W_A^{1/(2+\delta)}}\right)^k\right]\\
& = y^{-k} \sum_{i=M+1}^\infty \frac{\Gamma(i-\frac{1}{2+\delta})}{\Gamma(i+\frac{k-1}{2+\delta})}\frac{\Gamma(k+1+\delta)}{\Gamma(1+\delta)}\\
&\stackrel{M\to\infty}{\longrightarrow} 0,
\end{align*}
since $k/(2+\delta)>1$. This verifies Condition~(a).

To see \eqref{eq:meas_cond2}, we have
\begin{align*}
\left\{\sum_{i=1}^n \epsilon_{D^A_i(n)/n^{1/(2+\delta)}}(y,\infty] = 0\right\}
&= \left\{\frac{D^A_i(n)}{n^{1/(2+\delta)}}\le y, 1\le i\le n\right\}\\
&=\left\{\max_{1\le i\le n} \frac{D^A_i(n)}{n^{1/(2+\delta)}}\le y\right\}.
\end{align*}
Since we set $D^A_i(n) = 0$ for all $i\ge n+1$, then
\[
\left\{\max_{1\le i\le n} \frac{D^A_i(n)}{n^{1/(2+\delta)}}\le y\right\}
= \left\{\max_{i\ge 1} \frac{D^A_i(n)}{n^{1/(2+\delta)}}\le y\right\}.
\]
Similarly,
\[
\left\{\sum_{i=1}^\infty \epsilon_{\sigma_i e^{-T^A_i}/W_A^{1/(2+\delta)}}(y,\infty] = 0\right\}
= \left\{\max_{i\ge 1} \frac{\sigma_i e^{-T^A_i}}{W_A^{1/(2+\delta)}} \le y\right\}.
\]
By \eqref{eq:max_asyA}, we have for $y>0$,
\[
\PP\left(\max_{i\ge 1} \frac{D^A_i(n)}{n^{1/(2+\delta)}}\le y\right)\to \PP\left(\max_{i\ge 1} \frac{\sigma_i e^{-T^A_i}}{W_A^{1/(2+\delta)}} \le y\right),\quad \text{as }n\to\infty,
\]
which gives \eqref{eq:meas_cond2} and completes the proof of (iv).
%\end{proof}

\section{Consistency of Hill Estimator.}\label{sec:Hill}
\sid{We now turn to \eqref{eq:bnkD} as preparation for considering
consistency of the Hill estimator.} We first give a \sid{plausibility argument}
 based on the form of the limit point measure in \eqref{eq:meas_asyA}
or \eqref{eq:meas_asyB}. However,
proving \eqref{eq:bnkD} requires showing $N_{>k}(n)/n$ concentrates on
$p_{>k}$, for all $k\ge 1$, which in other words means controlling the
bias for $N_{>k}(n)/n$ and the discrepancy between $E(N_{>k}(n)/n)$
and $p_{>k}$.
Later we will show this is true for our Model A
but we were not successful for  Model B. See  Remark~\ref{rmk:modelB}.

\subsection{Heuristics.}
Before starting formalities, \sid{here is a} heuristic explanation for
the consistency of the Hill estimator \sid{when applied to preferential
attachment data} from Model A. The heuristic is the same for both
Model A and B so for simplicity, we focus on Model A \wtd{and apply} the Hill
estimator to the limit points in 
\eqref{eq:meas_asyA}. 
Since the Gamma random variables $\sigma_i$
have light tailed distributions, one may expect that $\{\sigma_i: i\ge
1\}$ will not distort the consistency result and so we pretend the
$\sigma_i$'s are absent; then what remains \wtd{in} the limit points is
monotone in $i$. Set $Y_i := e^{-T^A_i}/W_A^{1/(2+\delta)}$ and apply
the Hill estimator to the $Y's$ to get
\begin{align*}
H_{k,n} =&\frac 1k \sum_{i=1}^k \log \Bigl(\frac{Y_i}{Y_{k+1}}\Bigr)
=\frac 1k \sum_{i=1}^k( T^A_{k+1} -T^A_i).
\intertext{
Recall from just after \eqref{transprob1} that
$$T_{n+1}^A-T_n^A \stackrel{d}{=} E_n/(n(2+\delta)),$$ 
where $E_n, n\geq 1$ are iid unit exponential random variables. Then}
H_{k,n}=&\frac 1k \sum_{i=1}^k \sum_{l=i}^k(T^A_{l+1}-T_l^A)=\frac 1k
          \sum_{l=1}^k l(T_{l+1}^A-T^A_l)=\frac 1k \sum_{l=1}^k \frac{E_l}{2+\delta}\convas\frac{1}{2+\delta},
\end{align*}
by strong law of large numbers, provided that $k\to\infty$.
%Now $\sum_l \bigl( \Delta_l- 1/(l(2+\delta))\bigr)$ converges almost surely
%because the sum of the variances is finite. This implies
%$$\frac 1k \sum_{l=1}^k l\bigl(\Delta_l -1/(l(2+\delta))\bigr) \to 0,
%\qquad (k\to\infty)$$ and hence, provided $k\to\infty$,
%\begin{align*}
%\lim_{k\to\infty}H_{k,n}=& \lim_{k\to\infty} \frac 1k \sum_{l=1}^k
%                         l \Delta_l =\lim_{k\to\infty} \frac 1k \frac{kl}{l(2+\delta)}=\frac{1}{2+\delta}.
%\end{align*}

There are clear shortcomings to this approach, the most obvious being
that we 
only dealt with the points at asymptopia rather than $\{D_i(n), 1 \leq
i \leq n\}$. Furthermore we simplified the limit points by neglecting
the $\sigma_i$'s.
We have not found an effective
way to  analyze order statistics of $\{\sigma_i
  e^{-T^A_i}/W_A^{1/(2+\delta)}: i\ge 1\}$.

Concentration results for degree counts \sid{provide a traditional tool to}
prove \eqref{eq:bnkD} and we pursue this for \sid{Model A} in the next subsection.

\subsection{Concentration of the degree sequence in Model A}
We begin with considering the sequence of degree counts $\{N_{>k}(n)\}_{k\ge 1}$.
Theorem~\ref{thm:concentrate} shows that $N_{>k}(n)/n$ concentrates on $p_{>k}$, for all $k\ge 1$.
This concentration is what is needed for the consistency of the Hill estimator for network data.
Note that for the linear preferential attachment model, the concentration result for $N_k(n)$ is known from \citet[Theorem~8.3]{vanderHofstad:2017}.
%This draws an analogy to the concentration of $N_{k}(n)/n$ on $p_{k}$
%for fixed $k$ (cf. \cite[Theorem~8.3]{vanderHofstad:2017}). 

\begin{Theorem}\label{thm:concentrate}
For $\delta>-1$ there exists a constant $C>0$, such that as $n\to\infty$,
\beqq\label{eq:concentrate}
\PP\left(\max_k |N_{>k}(n) - np_{>k}|\ge C(1+\sqrt{n\log n})\right) = o(1).
\eeqq
\end{Theorem}
\begin{proof}
Let $\mu_{>k}(n) := \EE(N_{>k}(n))$.
Following the proof in \citet[Proposition~8.4]{vanderHofstad:2017}, we have for any $C_\mu>2\sqrt{2}$,
\[
\PP\left(|N_{>k}(n) - \mu_{>k}(n)|\ge C_\mu\sqrt{n\log n}\right) = o(1/n).
\]
Since $N_{>k}(n) = 0$ a.s. for all $k> n$, then
\begin{align}\label{eq:concen1}
\PP &\left(\max_k |N_{>k}(n) - \mu_{>k}(n)|\ge C_\mu\sqrt{n\log n}\right)\nonumber\\
=& \PP\left(\max_{0\le k\le n} |N_{>k}(n) - \mu_{>k}(n)|\ge C_\mu\sqrt{n\log n}\right)\nonumber\\
\le & \sum_{k=1}^n \PP\left(|N_{>k}(n) - \mu_{>k}(n)|\ge C_\mu\sqrt{n\log n}\right) = o(1).
\end{align} 
\wtd{Note that \eqref{eq:concen1} also holds for Model B, but we do not succeed in proving
the concentration result later in \eqref{eq:claim} for Model B; see Remark~\ref{rmk:modelB} for details.}
We are now left to show the concentration of $\mu_{>k}(n)$ on $n p_{>k}$ \wtd{in the setup of Model A}.
We claim that
\beqq\label{eq:claim}
%\sup_{k\ge 1} \sup_{n \ge 2} 
|\mu_{>k}(n) - n p_{>k}|\le C',\qquad \wtd{\forall n\ge 1,\quad\forall k\ge 1.}
\eeqq
for some constant $C'>0$ specified later.
We prove \eqref{eq:claim} by induction.
First, by model construction, $N_{>k} (n)$ satisfies 
\begin{align*}
\EE(N_{>k}(n+1)|G(n)) &= N_{>k}(n) + \frac{k+\delta}{(2+\delta)n} N_k(n) \\ 
&= N_{>k}(n) + \frac{k+\delta}{(2+\delta)n} (N_{>k-1}(n)-N_{>k}(n)), \qquad k \ge 1.
\end{align*}
Therefore,
\begin{align}\label{eq:mu}
\mu_{>k}(n+1) 
&= \mu_{>k}(n) + \frac{k+\delta}{(2+\delta)n} (\mu_{>k-1}(n)-\mu_{>k}(n)), \qquad k \ge 1.
\end{align}
Moreover, 
\wtd{it follows from \eqref{eq:pmfk} and \eqref{eq:def_pk} that
\[
p_{>k} = \frac{k+\delta}{2+\delta} p_k, \qquad k \ge 1.
\]
Thus}
$p_{>k}$ satisfies the recursion
\beqq\label{eq:pk}
p_{>k} = \frac{k+\delta}{2+\delta} (p_{>k-1}-p_{>k}), \qquad k \ge 1,
\eeqq
since $p_k = p_{>k-1}-p_{>k}$.
Let $\varepsilon_{>k}(n) := \mu_{>k}(n) - np_{>k}$, then \eqref{eq:mu} and \eqref{eq:pk} give that for $k\ge 1$,
\beqq\label{eq:diff}
\varepsilon_{>k}(n+1) = \left(1-\frac{k+\delta}{(2+\delta)n}\right) \varepsilon_{>k}(n) + \frac{k+\delta}{(2+\delta)n} \varepsilon_{>k-1}(n).
\eeqq

\wtd{In order to prove \eqref{eq:claim}, we initiate the induction procedure by
first inducting} on $n$ to prove %that \eqref{eq:claim} holds for $k = 1$, i.e. 
\beqq\label{eq:base2}
|\varepsilon_{>1}(n)|\le 1,\qquad\wtd{\forall n\ge 1}.
\eeqq
When $n = 1$, the graph $G(1)$ consists of one node and $D_1(1)  =2$. Since $p_{>k}\le 1$, we have
\beqq\label{eq:base1}
|\varepsilon_{>k}(1)| = |\mu_{>k}(1) -  p_{>k}| \le 1,\qquad \wtd{\forall k\ge 1},
\eeqq
\wtd{which also implies $|\varepsilon_{>1}(1)|\le 1$. Assume $|\varepsilon_{>1}(n)|\le 1$
and we want to show $|\varepsilon_{>1}(n+1)|\le 1$.}
Note that $ \varepsilon_{>0}(n) = \wtd{\mu_{>0}(n) - n p_{>0} = n-n\cdot 1} =0$, then \wtd{when $k=1$,} \eqref{eq:diff} becomes
\[
\varepsilon_{>1}(n+1) = \left(1-\frac{1+\delta}{(2+\delta)n}\right) \varepsilon_{>1}(n),
\]
and $1-\frac{1+\delta}{(2+\delta)n} \ge 0$ for $n\ge 1$.
This gives 
\[
|\varepsilon_{>1}(n+1)| \le \left(1-\frac{1+\delta}{(2+\delta)n}\right) |\varepsilon_{>1}(n)| \le 1.
\]
Hence, \eqref{eq:base2} is verified, \wtd{which gives the initialization step of the induction.}

\wtd{Since proving \eqref{eq:claim} requires showing 
\beqq\label{eq:hypo2}
\sup_{n\ge 1}\left|\varepsilon_{>k}(n)\right|\le C_p,\qquad \forall k\ge 1,
\eeqq
for some constant $C_p$ which will be defined later, we verify \eqref{eq:hypo2} by inducting on $k$.
What is proved in \eqref{eq:base2} gives the initialization of the induction ($k=1$) and
we want to verify
$$\left|\varepsilon_{>k}(n)\right|\le C_p,\qquad \forall n\ge 1,$$ 
assuming} 
\beqq\label{eq:inductn2}
\left|\varepsilon_{>k-1}(n)\right|\le C_p,\qquad \forall n\ge 1,
\eeqq
for some $k\ge 2$.
To do this, we again use induction on $n$, with the result for the base case $n=1$ being verified in \eqref{eq:base1}. 
\wtd{We now need to show $\left|\varepsilon_{>k}(n+1)\right|\le C_p$, given both $\left|\varepsilon_{>k}(n)\right|\le C_p$ and \eqref{eq:inductn2}.}

The recursion in \eqref{eq:diff} gives that for $2\le k\le (2+\delta)n - \delta$, 
\[
|\varepsilon_{>k}(n+1)| \le \left(1-\frac{k+\delta}{(2+\delta)n}\right) |\varepsilon_{>k}(n)| + \frac{k+\delta}{(2+\delta)n} |\varepsilon_{>k-1}(n)|
\le 1.
\]
For $k> (2+\delta)n - \delta$, $$|\varepsilon_{>k}(n+1)|  = (n+1)p_{>k}.$$
Since $(2+\delta)n - \delta \ge n+1$ for $\delta>-1$, $n\ge1$, \wtd{we apply \eqref{eq:def_pk} and} there exists a $C_p=C_p(\delta)$ such that
\[
p_{>k} \le C_p (n+1)^{-(2+\delta)},
\]
which gives 
\[
|\varepsilon_{>k}(n+1)|  = (n+1)p_{>k} \le C_p (n+1)^{-(1+\delta)}\le C_p.
\]
Thus, the claim in \eqref{eq:claim} is verified with 
\[
C':= \max\{1, C_p\}.
\]
\wtd{With the result in \eqref{eq:concen1}, the proof of the theorem is complete by choosing
$C = \max\{ C_\mu, C'\}$.}
\end{proof}
\begin{Remark}\label{rmk:modelB} {\rm
The induction argument does not suffice to prove
 \eqref{eq:claim} for Model B.
To see this, we re-compute the recursion on the difference term $\varepsilon_{>k}(n)$ for Model B and \eqref{eq:diff} then becomes
\begin{align*}
\varepsilon_{>k}(n+1) =& \left(1-\frac{k+\delta}{(2+\delta)n+1+\delta}\right) \varepsilon_{>k}(n) + \frac{k+\delta}{(2+\delta)n+1+\delta} \varepsilon_{>k-1}(n)\\
&+\left(\frac{1}{2+\delta}-\frac{n}{(2+\delta)n+1+\delta}\right) (k+\delta) (p_{>k-1}-p_{>k})\\
=& \left(1-\frac{k+\delta}{(2+\delta)n+1+\delta}\right) \varepsilon_{>k}(n) + \frac{k+\delta}{(2+\delta)n+1+\delta} \varepsilon_{>k-1}(n)\\
&+\frac{1+\delta}{2+\delta}\frac{(k+\delta)p_k}{(2+\delta)n+1+\delta}.
\end{align*}
By \citet[Exercise 8.19]{vanderHofstad:2017}, $(k+\delta)p_k\le 2+\delta$.
Therefore, if $|\varepsilon_{>k}(n)|\le 1$, then
\begin{align*}
|\varepsilon_{>k}(n+1)|\le 1+ \frac{1}{n+1},
\end{align*}
which contradicts the induction hypothesis.

Since the concentration inequality proved in Theorem~\ref{thm:concentrate} cannot be validated for Model B by induction,
we are also not able to verify the consistency of the Hill estimator in Model B, using the proof steps proposed here. This
will be deferred as future research.}
\end{Remark}

\subsection{Convergence of the tail empirical measure for Model A}
We then use the concentration result in \eqref{eq:concentrate} to analyze the convergence of the tail empirical measure.
First consider the degree of each node in $G(n)$,
\[
(D_1(n),D_2(n),\ldots, D_n(n)),
\]
and let 
\[
D_{(1)}(n) \ge D_{(2)}(n) \ge \cdots \ge D_{(n)}(n)
\]
be the corresponding order statistics. Then the tail empirical measure becomes
\[
\hat{\nu}_n (\cdot) := \frac{1}{k_n} \sum_{i=1}^n \epsilon_{D_i(n)/D_{(k_n)}(n)}(\cdot),
\]
for some intermediate sequence $\{k_n\}$, i.e. $k_n\to\infty$ and $k_n/n\to 0$ as $n\to\infty$.
\begin{Theorem}
Suppose that $\{k_n\}$ is some intermediate sequence satisfying 
\beqq\label{cond:kn}
\liminf_{n\to\infty} k_n/(n\log n)^{1/2}>0\quad \text{and}\quad k_n/n\to 0 \quad \text{as}\quad n\to\infty,
\eeqq
 then
\beqq\label{eq:tailmeas}
\hat{\nu}_n \Rightarrow \nu_{2+\delta},
\eeqq
in $M_+((0,\infty])$, where $\nu_{2+\delta}(x,\infty] = x^{-(2+\delta)}$, $x>0$.
\end{Theorem}
\begin{proof}
\emph{Step 1.} We first show that for fixed $t>0$,
\beqq\label{eq:step1}
\frac{D_{([k_n t])}(n)}{b(n/k_n)} \convp t^{-\frac{1}{2+\delta}},
\eeqq
where 
$$ b(n/k_n) = \left(\frac{\Gamma(3+2\delta)}{\Gamma(1+\delta)}\right)^{\frac{1}{2+\delta}}(n/k_n)^{\frac{1}{2+\delta}}.$$
Since
\begin{align*}
\PP&\left(\left|\frac{D_{([k_n t])}(n)}{b(n/k_n)}- t^{-\frac{1}{2+\delta}}\right|>\epsilon\right)\\
\le & \PP(D_{([k_n t])}(n) > b(n/k_n) (t^{-\frac{1}{2+\delta}}+\epsilon)) + \PP(D_{([k_n t])}(n)< b(n/k_n)(t^{-\frac{1}{2+\delta}}-\epsilon))\\
=: &\, I+II,
\end{align*}
it suffices to show $I,II\to 0$ as $n\to\infty$.
For the first term, we have, with $ u_t := t^{-\frac{1}{2+\delta}}+\epsilon$,
\begin{align}
I &\le \PP(N_{>[b(n/k_n)u_t]}(n)\ge [k_n t])\nonumber\\
&= \PP\left(N_{>[b(n/k_n)u_t]}(n) - n p_{>[b(n/k_n)u_t]}
 \ge [k_n t]- n p_{>[b(n/k_n)u_t]}\right).\label{eq:part1}
\end{align}
Using Stirling's formula, \citet[Equation~8.3.9]{vanderHofstad:2017} gives 
\[
\frac{\Gamma(t+a)}{\Gamma(t)} = t^a(1+O(1/t)).
\]
Recall the definition of $p_{>k}$ in \eqref{eq:def_pk} for fixed $k$, then we have %as $n\to\infty$,
\begin{align}\label{eq:conv_poverk}
\frac{n}{k_n} p_{>[b(n/k_n) y]} &= \frac{n}{k_n} \frac{\Gamma(3+2\delta)}{\Gamma(1+\delta)}\frac{\Gamma\left([b(n/k_n) y]+1+\delta\right)}{\Gamma\left([b(n/k_n) y]+3+2\delta\right)}\nonumber\\
&= \frac{\Gamma(3+2\delta)}{\Gamma(1+\delta)} \frac{n}{k_n} \left(b(n/k_n)y\right)^{-(2+\delta)}\left(1+O\left(\frac{1}{b(n/k_n)}\right)\right)\nonumber\\
&= y^{-(2+\delta)} \left(1+O\left(\frac{1}{b(n/k_n)}\right)\right). 
%= y^{-(2+\delta)} \left(1+o(1)\right). 
%\longrightarrow y^{-(2+\delta)}, \qquad \text{as }n\to\infty.
\end{align}
Continuing from \eqref{eq:part1} then gives
\begin{align*}
I &\le 
\PP\left(N_{>[b(n/k_n)u_t]}(n) - n p_{>[b(n/k_n)u_t]}
 \ge [k_n t]- n p_{>[b(n/k_n)u_t]}\right)\\
\le & \PP\left(\left|N_{>[b(n/k_n)u_t]}(n) - n p_{>[b(n/k_n)u_t]}\right|
 \ge k_n\left(t - u_t^{-(2+\delta)}+O(1/b(n/k_n))\right)\right)\\
 \le & \PP\left(\max_j\left|N_{>j}(n) - n p_{>j}\right|
 \ge k_n\left(t - u_t^{-(2+\delta)}+O(1/b(n/k_n))\right)\right).
\end{align*}
By Theorem~\ref{thm:concentrate}, the right hand side goes to 0 as $n\to\infty$, provided that $k_n$ satisfies \eqref{cond:kn}.
Similarly, we can also show $II\to 0$ as $n\to\infty$ for $k_n$ satisfying \eqref{cond:kn}, thus proving \eqref{eq:step1}.

\emph{Step 2.} Note that $D_{([k_n t])}(n)$ is decreasing in $t$ and the limit in \eqref{eq:step1} is continuous on $(0,\infty]$,
which implies
\[
\frac{D_{([k_n t])}(n)}{b(n/k_n)} \convp t^{-\frac{1}{2+\delta}}, \quad \text{in } D(0,\infty].
\]
This gives, by inversion and \citet[Proposition~3.2]{resnickbook:2007},
\beqq\label{eq:step2}
\frac{1}{k_n} \sum_{i=1}^n \epsilon_{D_i(n)/b(n/k_n)}(t, \infty] \convp t^{-(2+\delta)}, \quad t\in (0,\infty],
\eeqq
in $D(0,\infty]$. Moreover,
\beqq\label{eq:jointconv}
\left(\frac{1}{k_n} \sum_{i=1}^n \epsilon_{D_i(n)/b(n/k_n)}, \frac{D_{([k_n])}(n)}{b(n/k_n)}\right)
\Rightarrow \left(\nu_{2+\delta},1\right)
\eeqq
in $M_+((0,\infty])\times (0,\infty)$.

\emph{Step 3.} With \eqref{eq:step2}, we use a scaling argument to prove \eqref{eq:tailmeas}. Define the operator
\[
S: M_+((0,\infty])\times (0,\infty) \mapsto M_+((0,\infty])
\]
by 
\[
S(\nu, c)(A) = \nu(cA).
\]
By the proof in \citet[Theorem~4.2]{resnickbook:2007}, the mapping $S$ is continuous at $(\nu_{2+\delta},1)$.
Therefore, applying the continuous mapping $S$ to the joint weak convergence in \eqref{eq:jointconv} gives \eqref{eq:tailmeas}.
\end{proof}

\subsection{Consistency of the Hill estimator for Model A}
We are now able to prove the consistency of the Hill estimator applied to $\{D_i(n): 1\le i\le n\}$, i.e.
\[
H_{k_n, n} = \frac{1}{k_n}\sum_{i=1}^{k_n} \log\frac{D_{(i)}(n)}{D_{(k_n+1)}(n)}.
\]
\begin{Theorem}
Let $\{k_n\}$ be an intermediate sequence satisfying \eqref{cond:kn}, then
\[
H_{k_n, n} \convp \frac{1}{2+\delta}.
\]
\end{Theorem}
\begin{proof}
First define a mapping $T: D(0,\infty] \mapsto \RR_+$ by
\[
T(f) = \int_1^\infty f(y)\frac{\dd y}{y},
\]
and note that 
\[
H_{k_n, n} = \int_1^\infty \hat{\nu}_n(y,\infty] \frac{\dd y}{y}.
\]
Therefore, proving the consistency of $H_{k_n, n}$ requires justifying the continuity of the mapping $T$ at $\nu_{2+\delta}$, so that
\[
H_{k_n, n} = \int_1^\infty \hat{\nu}_n(y,\infty] \frac{\dd y}{y}\convp \int_1^\infty \nu_{2+\delta}(y,\infty] \frac{\dd y}{y} =\frac{1}{2+\delta}.
\]
Note that for any $M$ we have
\[
\int_1^M \hat{\nu}_n(y,\infty] \frac{\dd y}{y}\convp \int_1^M \nu_{2+\delta}(y,\infty] \frac{\dd y}{y},
\]
so we only need to show
\[
\int_M^\infty \hat{\nu}_n(y,\infty] \frac{\dd y}{y}\convp \int_M^\infty \nu_{2+\delta}(y,\infty] \frac{\dd y}{y}.
\]

By the second converging together theorem (see \citet[Theorem~3.5]{resnickbook:2007}), 
it suffices to show
\beqq\label{eq:cont_integral}
\lim_{M\to\infty} \limsup_{n\to\infty} \PP\left(\int_M^\infty \hat{\nu}_n(y,\infty] \frac{\dd y}{y}>\varepsilon\right) = 0.
\eeqq
Consider the probability in \eqref{eq:cont_integral} and \wtd{we have}
\begin{align*}
\PP & \left(\int_M^\infty \hat{\nu}_n(y,\infty] \frac{\dd y}{y}>\varepsilon\right)\\
\le\, &\PP\left(\int_M^\infty \hat{\nu}_n(y,\infty] \frac{\dd y}{y}>\varepsilon, \left|\frac{D_{(k_n)}(n)}{b(n/k_n)}-1\right|<\eta\right)\\
&+ \PP\left(\int_M^\infty \hat{\nu}_n(y,\infty] \frac{\dd y}{y}>\varepsilon, \left|\frac{D_{(k_n)}(n)}{b(n/k_n)}-1\right|\ge\eta\right)\\
\le\, \PP & \left(\int_M^\infty \frac{1}{k_n}\sum_{i=1}^n \epsilon_{D_i(n)/b(n/k_n)}((1-\eta)y,\infty] \frac{\dd y}{y}>\varepsilon\right)\\
 +& \PP\left(\left|\frac{D_{(k_n)}(n)}{b(n/k_n)}-1\right|\ge\eta\right) =: A+B.
\end{align*}
By \eqref{eq:step1}, $B\to 0$ as $n\to\infty$, and using the Markov inequality, $A$ is bounded by
\begin{align*}
\frac{1}{\varepsilon} &\EE\left(\int_M^\infty \frac{1}{k_n}\sum_{i=1}^n \epsilon_{D_i(n)/b(n/k_n)}((1-\eta)y,\infty] \frac{\dd y}{y}\right)\\
&= \frac{1}{\varepsilon}\EE\left(\int_{M(1-\eta)}^\infty \frac{1}{k_n}\sum_{i=1}^n \epsilon_{D_i(n)/b(n/k_n)}(y,\infty] \frac{\dd y}{y}\right)\\
&\le \frac{1}{\varepsilon}\int_{M(1-\eta)}^\infty \frac{1}{k_n}\EE\left(N_{>[b(n/k_n) y]}(n)\right) \frac{\dd y}{y}.
\end{align*}
Furthermore, we also have for $y>0$,
\begin{align*}
\Big|\frac{1}{k_n}  &\EE\left(N_{>[b(n/k_n) y]}(n)\right)-y^{-(2+\delta)}\Big|\\
\le  &\frac{1}{k_n}\Big|\EE\left(N_{>[b(n/k_n) y]}(n)\right)- n p_{>[b(n/k_n) y]}\Big|
 + \left|\frac{n}{k_n} p_{>[b(n/k_n) y]} - y^{-(2+\delta)}\right|.
\end{align*}
According to \eqref{eq:claim}, the first term is bounded above by $C'/k_n\to 0$ as $n\to\infty$. 
The second term also goes to 0 by \eqref{eq:conv_poverk} as $n\to\infty$.
Hence, as $n\to\infty$,
\beqq\label{eq:conv_Eoverk}
\frac{1}{k_n}\EE\left(N_{>[b(n/k_n) y]}(n)\right)\rightarrow y^{-(2+\delta)}.
\eeqq
Let $U(t) := \EE(N_{>[t]}(n))$ and \eqref{eq:conv_Eoverk} becomes: for $y>0$,
\[
\frac{1}{k_n} U(b(n/k_n) y) \rightarrow y^{-(2+\delta)},\quad \text{as }n\to\infty.
\]
Since $U(\cdot)$ is a non-increasing function, $U\in RV_{-(2+\delta)}$ by \citet[Proposition 2.3(ii)]{resnickbook:2007}.
Therefore, Karamata's theorem gives
\[
A \le \frac{1}{\varepsilon}\int_{M(1-\eta)}^\infty \frac{1}{k_n}\EE\left(N_{>[b(n/k_n) y]}(n)\right) \frac{\dd y}{y}
%\stackrel{n\to\infty}{\longrightarrow} 
\sim C(\delta, \eta) M^{-(2+\delta)},
\]
with some positive constant $C(\delta, \eta)>0$.
Also, $M^{-(2+\delta)}\to 0$ as $M\to\infty$, and \eqref{eq:cont_integral} follows.
\end{proof}

\section{Simulation Studies.}\label{sec:sim}
As noted in Remark~\ref{rmk:modelB}, we \wtd{fail to} prove the consistency of the Hill estimator in Model B 
using the techniques \sid{of} Section~\ref{sec:Hill}. 
In this section, however, we give some simulation results to see how consistent the Hill estimator is in Model B.

The main problem is to choose a proper $k_n$. We adopt the threshold
selection method proposed in \cite{clauset:shalizi:newman:2009},  
which is also widely used in online data sources like KONECT
\cite{kunegis:2013}. \wtd{This
method is encoded in %the R-package {\it poweRlaw\/} \cite{gillespie:2015}.
the \texttt{plfit} script, which can be found at \url{http://tuvalu.santafe.edu/~aaronc/powerlaws/plfit.r)}.}
Here is a summary of this method that we refer to it as the ``minimum
distance method".
Given a sample of $n$ iid observations, $Z_1,\ldots, Z_n$ from a power
law distribution with tail index $\alpha$, the minimum distance method suggests using
the thresholded data consisting of the $k$ upper-order statistics, $Z_{(1)}\ge\ldots \ge Z_{(k)}$, for estimating $\alpha$. 
The tail index is estimated by
\[
\hat{\alpha}(k):= \left( \frac{1}{k}\sum_{i=1}^{k} \log\frac{Z_{(i)}}{Z_{(k+1)}} \right)^{-1}, \quad k\ge 1.
\]
To select $k$, we first compute
 the Kolmogorov-Smirnov (KS) distance between the
empirical tail distribution of the upper $k$ observations and the power-law
tail with index $\hat{\alpha}(k)$:
\[
d_k:=\supy \left|\frac{1}{k}\sum_{i=1}^n\epsilon_{Z_i/Z_{(k+1)}}(y,\infty]-y^{-\hat{\alpha}(k)}\right|, \quad 1\le k\le n.
\]
Then the optimal $k^*$ is  the one that minimizes the KS distance, i.e.
$$
k^* := \argmin_{1\le k\le n}\, d_k,
$$
and we estimate the tail index and threshold by $\hat{\alpha}(k^*)$ and \wtd{$Z_{(k^*+1)}$} respectively.
This estimator performs well if the thresholded portion
comes from a Pareto tail and also seems effective in a
variety of non-iid scenarios.
  
We chose $\delta = -0.5, 0, 0.5, 1, 2$ then the theoretical tail indices of degree distributions from Model B were equal to $\alpha:=2+\delta=1.5, 2, 2.5, 3, 4$, respectively.
For each value of $\delta$, we also varied the number of edges in the network:
$n=5000, 10000, 50000, 100000$. 
For each combination of $(\alpha, n)$, we \sid{simulated} 500 independent
replications of the preferential attachment network using
\sid{software discussed in \cite{wan:wang:davis:resnick:2017} and
linked to
\url{http://www.orie.cornell.edu/orie/research/groups/multheavytail/software.cfm}.}
For each replication we
computed $\hat{\alpha}(k^*)$  using the minimum distance method. We
recorded the mean of those 500 estimates in the corresponding entry of
Table~\ref{Table:alpha_star}, based on the combination of $(\alpha,
n)$. 
\begin{table}[h]
\centering
\begin{tabular}{l|cccc}
\hline
& \multicolumn{4}{c}{Number of Edges} \\
 & 5000 & 10000 & 50000 & 100000 \\ 
\hline 
$\alpha = 1.5$ & 1.481 & 1.484 & 1.484 & 1.488 \\ 
$\alpha = 2$ & 2.061 & 2.028 & 1.998 & 1.990\\ 
$\alpha = 2.5$ & 2.602 & 2.557 & 2.507 & 2.494 \\ 
$\alpha = 3$ & 3.135 & 3.079 & 3.045 & 2.983 \\ 
$\alpha = 4$ & 3.957 & 3.930 & 3.942 & 3.932 \\ 
\hline
\end{tabular} 
\caption{Mean values of $\hat{\alpha}(k^*)$ over 500 estimates using the minimum distance method, for each combination of $(\alpha, n)$.}
\label{Table:alpha_star}
\end{table}

We see that when $\delta = -0.5 <0 $, i.e. $\alpha = 1.5$, the minimum distance estimate $\hat{\alpha}(k^*)$
consistently underestimates the tail index, even if the number of edges in the network has been increased to $10^5$.
For the cases where $\delta\ge 0$ (i.e. $\alpha\ge 2$), the tail estimates have smaller biases as $n$ increases, as long as 
the tails are not too ``light". When $\alpha = 4$, the tail becomes much lighter. Because of the finite sample bias that may occur while applying the minimum distance method to lighter-tailed power laws, increasing the number of
edges in the network does not significantly improve the bias of estimates.

\begin{figure}[h]
\centering
\includegraphics[scale=.38]{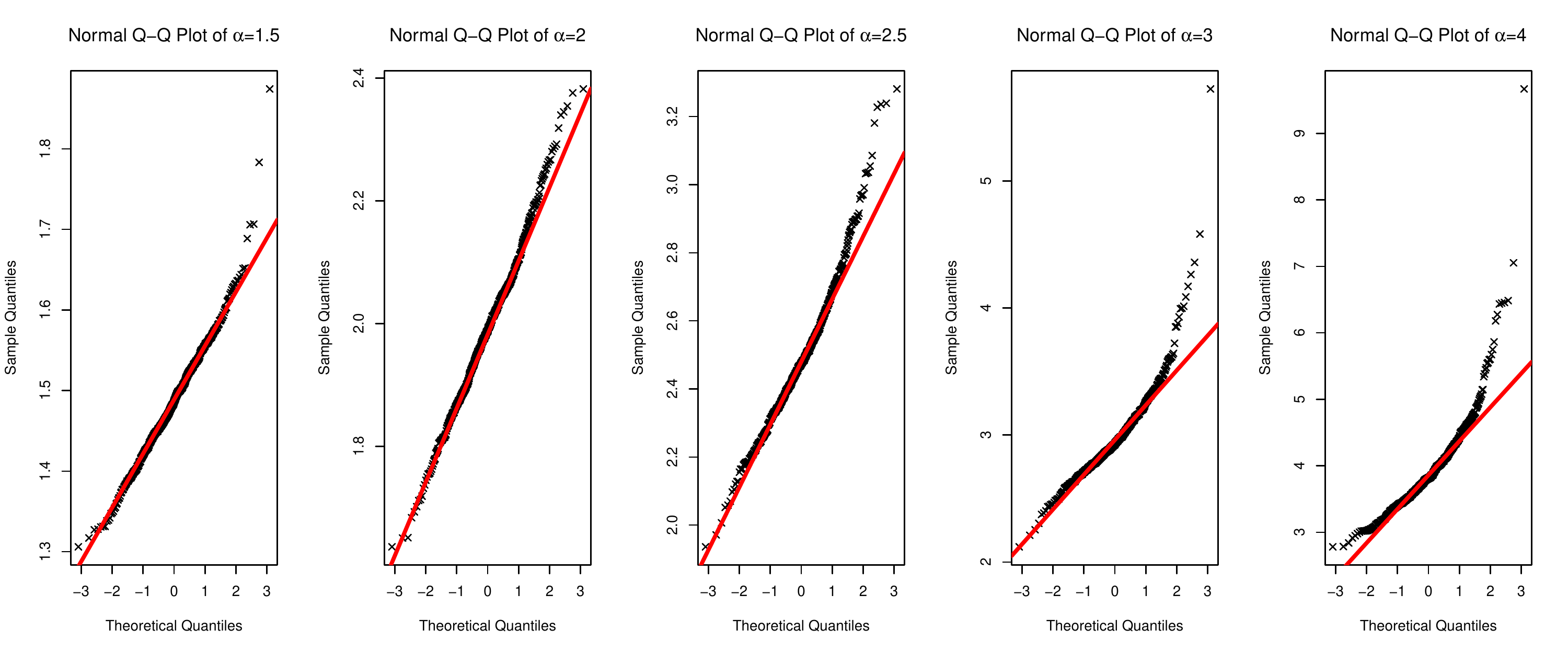}
\caption{QQ plots of $\hat{\alpha}(k^*)$ with $n=10^5$ and $\alpha = 1.5, 2, 2.5, 3, 4$. The fitted lines in red are the traditional qq-lines used to check normality of the estimates.}\label{fig:QQ}
\end{figure}
In Figure~\ref{fig:QQ}, we provide the QQ plots of those 500 minimum distance tail estimates $\hat{\alpha}(k^*)$ while holding $n = 10^5$ and varying $\alpha$ as specified in Table~\ref{Table:alpha_star}. The fitted lines in red are the traditional qq-lines used to check normality of the estimates.
When $\delta\le 0$ (i.e. the cases where $\alpha = 1.5, 2$), QQ plots
\sid{are consistent with} normality of $\hat{\alpha}(k^*)$. 
However, as $\delta$ increases ($\alpha = 2.5, 3, 4$), significant
departures from the normal distribution are observed and asymptotic
normality is \sid{not proven theoretically or empirically.}

In conclusion, for Model B, simulation results \sid{suggest} that the Hill estimator is consistent when $\delta\ge 0$ (i.e. the tail index $\alpha\ge 2$),
but the asymptotic normality is not guaranteed. Since we only have QQ plots of the minimum distance estimates in Figure~\ref{fig:QQ}, it is still not clear whether this non-normality is due to the minimum distance method or the dependence in the network data.
We intend to analyze further the consistency for Model B and other
variants, \wtd{as well as the} asymptotic behavior of the Hill estimator.
\bibliography{./bibfile}
\end{document}